\documentclass[12pt]{article}
\usepackage{amssymb,latexsym,amsfonts,amsthm,verbatim,amscd,ifthen,amsmath,graphicx,color} 
\DeclareMathAlphabet{\mathpzc}{OT1}{pzc}{m}{it}
  
\begin{document} 
\title{A Minimal Poset Resolution of Stable Ideals}
\author{Timothy B.P. Clark}
\date{\today}

\newtheorem{intheorem}{Theorem} 
\newtheorem{inlemma}[intheorem]{Lemma} 
\newtheorem{inproposition}[intheorem]{Proposition} 
\newtheorem{incorollary}[intheorem]{Corollary}
\newtheorem{indefinition}[intheorem]{Definition} 
\newtheorem{inremark}[intheorem]{Remark} 
\newtheorem{innotation}[intheorem]{Notation}

\newtheorem{theorem}[equation]{Theorem} 
\newtheorem{lemma}[equation]{Lemma} 
\newtheorem{proposition}[equation]{Proposition} 
\newtheorem{corollary}[equation]{Corollary} 
\newtheorem{observation}[equation]{Observation} 
\newtheorem{conjecture}[equation]{Conjecture} 
\newtheorem{fact}[equation]{Fact} 
\newtheorem{example}[equation]{Example} 
\newtheorem{definition}[equation]{Definition} 
\newtheorem{remark}[equation]{Remark} 
\newtheorem{remarks}[equation]{Remarks} 
\newtheorem{notation}[equation]{Notation} 
\newtheorem*{acknowledgements}{Acknowledgements}

\newenvironment{indented}{\begin{list}{}{}\item[]}{\end{list}} 

\renewcommand{\:}{\! :\ } 
\newcommand{\p}{\mathfrak p} 
\newcommand{\m}{\mathfrak m}
\newcommand{\e}{\epsilon} 
\newcommand{\g}{{\bf g}} 
\newcommand{\lra}{\longrightarrow} 
\newcommand{\ra}{\rightarrow} 
\newcommand{\altref}[1]{{\upshape(\ref{#1})}} 
\newcommand{\bfa}{\mathbf a} 
\newcommand{\bfb}{\boldsymbol{\beta}} 
\newcommand{\bfg}{\boldsymbol{\gamma}} 
\newcommand{\bfd}{\boldsymbol{\delta}} 
\newcommand{\bfM}{\mathbf M} 
\newcommand{\bfN}{\mathbf N}
\newcommand{\bfI}{\mathbf I} 
\newcommand{\bfC}{\mathbf C} 
\newcommand{\bfB}{\mathbf B} 
\newcommand{\mcP}{\mathcal P}
\newcommand{\mcX}{\mathcal X}
\newcommand{\D}{\textbf{\textup{D}}}
\newcommand{\bsfC}{\bold{\mathsf C}} 
\newcommand{\bsfT}{\bold{\mathsf T}}
\newcommand{\mc}{\mathcal} 
\newcommand{\smsm}{\smallsetminus} 
\newcommand{\ol}{\overline} 
\newcommand{\twedge}
           {\smash{\overset{\mbox{}_{\circ}}
                           {\wedge}}\thinspace} 
\newcommand{\mbb}[1]{\mathbb{#1}}
\newcommand{\pring}{\Bbbk[x_1,\ldots,x_d]}
\newcommand{\irr}{(x_1,\ldots,x_d)}
\newcommand{\Z}{\textup{Z}}
\newcommand{\B}{\textup{B}}
\newcommand{\La}{\mathcal{L}}
\newcommand{\redHo}{\widetilde{\textup{H}}}
\newcommand{\redH}[2]{\widetilde{\textup{H}}_{#1}(#2)}
\newcommand{\pr}{\textup{proj}}
\newcommand{\precdot}{\prec\!\!\!\cdot\;}
\newcommand{\succdot}{\;\cdot\!\!\!\succ}
\newcommand{\vset}[2]{\textbf{V}_{{#1},{#2}}}
\newcommand{\sx}[2]{\textbf{B}_{{#1},{#2}}}
\newcommand{\cplx}[2]{\Delta_{{#1},{#2}}}
\newcommand{\vx}[1]{(\varnothing,#1)}
\newcommand{\ds}{\displaystyle}
\newcommand{\Sy}{\Sigma}
\newcommand{\Ho}{\widetilde{H}}
\newcommand{\sgn}{\textup{sgn}}
\newcommand{\CC}{\widetilde{\mathcal{C}}_\bullet}
\newcommand{\ld}{\lessdot}
\newcommand{\mdeg}{\textup{mdeg}}
\newcommand{\id}{\textup{id}}

\newlength{\wdtha}
\newlength{\wdthb}
\newlength{\wdthc}
\newlength{\wdthd}
\newcommand{\elabel}[1]
           {\label{#1}  
            \setlength{\wdtha}{.4\marginparwidth}
            \settowidth{\wdthb}{\tt\small{#1}} 
            \addtolength{\wdthb}{\wdtha}
            \smash{
            \raisebox{.8\baselineskip}
            {\color{red}
             \hspace*{-\wdthb}\tt\small{#1}\hspace{\wdtha}}}}

\newcommand{\mlabel}[1] 
           {\label{#1} 
            \setlength{\wdtha}{\textwidth}
            \setlength{\wdthb}{\wdtha} 
            \addtolength{\wdthb}{\marginparsep} 
            \addtolength{\wdthb}{\marginparwidth}
            \setlength{\wdthc}{\marginparwidth}
            \setlength{\wdthd}{\marginparsep}
            \addtolength{\wdtha}{2\wdthc}
            \addtolength{\wdtha}{2\marginparsep} 
            \setlength{\marginparwidth}{\wdtha}
            \setlength{\marginparsep}{-\wdthb} 
            \setlength{\wdtha}{\wdthc} 
            \addtolength{\wdtha}{1.1ex}
            \marginpar{\vspace*{-0.3\baselineskip}
                       \tt\small{#1}\\[-0.4\baselineskip]\rule{\wdtha}{.5pt} }
            \setlength{\marginparwidth}{\wdthc} 
            \setlength{\marginparsep}{\wdthd}  }  
            
\renewcommand{\mlabel}{\label} 
\renewcommand{\elabel}{\label} 

\newcommand{\mysection}[1]
{\section{#1}\setcounter{equation}{0}
             \numberwithin{equation}{section}}

\newcommand{\mysubsection}[1]
{\subsection{#1}\setcounter{equation}{0}
                \numberwithin{equation}{subsection}}

\newcommand{\mysubsubsection}[1]
{\subsubsection{#1}\setcounter{equation}{0}
                \numberwithin{equation}{subsubsection}}

\maketitle

\begin{abstract}
We use the theory of poset resolutions to give an alternate 
construction for the minimal 
free resolution of an arbitrary stable monomial ideal in the 
polynomial ring whose coefficients are from a field.  
This resolution is recovered by utilizing a poset of 
Eliahou-Kervaire admissible symbols associated to a stable ideal.  
The structure of the poset under consideration is quite rich and 
in related analysis, we exhibit a regular CW complex which supports 
a minimal cellular resolution of a stable monomial ideal.  
\end{abstract}

\mysection{Introduction}
Let $R=\pring$, where $\Bbbk$ is a field and $R$ is 
considered with its standard $\mbb{Z}^d$ grading (multigrading).  
A monomial ideal $N\subseteq R$ is called \emph{stable} 
if for every monomial $m\in{N}$, the monomial $m\cdot{x_i}/{x_r}\in{N}$ 
for each $1\le{i}<r$, where $r=\max\{k:x_k\textrm{ divides }m\}$.  
The class of stable ideals, introduced by Eliahou and Kervaire in 
\cite{EK} is arguably the most well-studied class of monomial ideals.  
In their work, Eliahou and Kervaire construct the minimal free 
resolution of an arbitrary stable monomial ideal by identifying 
basis elements of the free modules present in the resolution 
-- which they call admissible symbols -- and by describing how the 
maps within the resolution act on these admissible symbols.  

Among others, the class of stable ideals contains the so-called 
Borel-fixed \cite{Eis} ideals as a subclass, which are fundamental 
to Gr\"obner Basis Theory.  Certain homological properties of 
stable ideals are analyzed by Herzog and Hibi \cite{HH}, where 
it is shown that stable ideals are componentwise linear.  More recently, 
distinct topological methods have been used to describe the 
minimal free resolution of an arbitrary stable ideal 
\cite{BatziesWelker,Mermin}.  

We construct the minimal free resolution of an arbitrary stable 
ideal $N$ using the theory of \emph{poset resolutions}, introduced 
by the author in \cite{Clark}.  Specifically, we define a poset 
$(P_N,<)$ on the admissible symbols of Eliahou and Kervaire by 
taking advantage of a decomposition property unique to the monomials 
contained in stable ideals.  In our first main result, Theorem 
\ref{stableTheorem}, we recover the Eliahou-Kervaire minimal free resolution of a 
stable ideal as a poset resolution.  The value of this technique 
is that the maps in the resolution are shown to act on the basis 
elements of the free modules in a way that mirrors the covering 
relations present in $P_N$.  Taking the lattice-linear ideals of 
\cite{Clark} into consideration, poset resolutions therefore provide a 
common perspective from which to view the minimal resolutions 
of three large and well-studied classes of monomial ideals; 
stable ideals, Scarf ideals \cite{BPS} and ideals having a 
linear resolution \cite{EisGS,HH}.  

An additional advantage of the method described herein 
is that for a fixed stable ideal, 
combinatorial information from the poset of admissible 
symbols is transferred to topological information of a CW complex.  
Specifically, the poset $P_N$ is a \emph{CW poset} in 
the sense of Bj\"orner \cite{BjCW}, which allows its 
interpretation as the face poset of a regular CW complex $X_N$.  
In our second main result, Theorem \ref{StableCellularTheorem}, 
we show that $X_N$ supports a minimal cellular resolution of 
the stable ideal $N$.  By utilizing this combinatorial 
connection, we are able to provide a minimal cellular resolution 
of $N$ in a considerably more explicit manner than in three 
previous methods; the construction of Batzies and Welker 
\cite{BatziesWelker} which uses the tools of Discrete Morse Theory, 
the result due to Mermin \cite{Mermin} that the Eliahou-Kervaire 
resolution is cellular and the technique of Sinefakopoulos 
\cite{AS} which produces a polyhedral cell complex that 
supports the minimal free resolution of (the more restricted 
class of) Borel-fixed ideals generated in one degree.  In addition, 
Corso and Nagel in \cite{CorsoNagel}, construct cellular resolutions of strongly stable ideals 
generated in degree two and separately, Nagel and Reiner in \cite{NagelReiner}, construct 
cellular resolutions of strongly stable ideals generated in one degree.
To the author's knowledge, connections between these methods and the one 
contained in this paper have not been formally studied.  

In Section \ref{PRes}, we review the fundamentals of poset 
resolutions and define the poset of admissible symbols $P_N$.  
Section \ref{PNShellable} provides the details of the (dual) shellability 
of $P_N$.  The intrinsic CW poset structure and properties of 
the sequence of vector spaces and maps which is produced by the 
methods of \cite{Clark} are exhibited in Section \ref{PropertiesOfCPN}.  
The proof of Theorem \ref{stableTheorem} is given in Section 
\ref{PRStableProof} and Section \ref{StableCellular} contains 
the description of the regular CW complex $X_N$ and the 
proof of Theorem \ref{StableCellularTheorem}.

Throughout the paper, we assume that the topolological and 
combinatorial notions of posets, order complexes, CW complexes 
and face posets are familiar to the reader.  

\begin{acknowledgements}
Many thanks to Jeff Mermin, whose comments 
were of considerable help in improving the clarity 
of the exposition and in correcting an omission 
of certain cases which appear in the proofs of 
Theorem \ref{PNshellable} and Lemma \ref{CycleLemma}.  
Thanks as well to Amanda Beecher and Alexandre 
Tchernev for helpful questions and discussions.  
\end{acknowledgements}

\mysection{Poset Resolutions and Stable Ideals}\label{PRes}
Let $(P,<)$ be a finite poset with set of atoms $A$ and write 
$\beta\ld\alpha$ if $\beta<\alpha$ and there is no $\gamma\in P$ 
such that $\beta<\gamma<\alpha$.  We say that $\beta$ is 
\emph{covered by} $\alpha$ in this situation.  For $\alpha\in P$, 
write the order complex of the associated open interval as 
$\Delta_\alpha=\Delta(\hat{0},\alpha)$.  In \cite{Clark}, the 
collection of simplicial complexes $$\{\Delta_\alpha:\alpha\in{P}\}$$ 
is used to construct a sequence of vector spaces and vector space maps 
$$
\mathcal{D}(P):\cdots\lra{\mathcal{D}_i}
\stackrel{\varphi_i}{\lra}\mathcal{D}_{i-1}\lra
\cdots\lra{\mathcal{D}_1}\stackrel{\varphi_1}{\lra}\mathcal{D}_0.
$$ 
In this construction, $\mathcal{D}_0=\Ho_{-1}(\{\varnothing\},\Bbbk)$, 
a one-dimensional vector space.  

For $i\ge{1}$, the vector space $\mathcal{D}_i$ is defined as 
$$\mathcal{D}_i=\ds \bigoplus_{\alpha\in{P}\smsm\{\hat{0}\}}\mathcal{D}_{i,\alpha},$$
where $\mathcal{D}_{i,\alpha}=\Ho_{i-2}(\Delta_\alpha,\Bbbk)$.  
In particular, the vector space $\mathcal{D}_1$ has its nontrivial components indexed by the 
set of atoms $A$, and the map $\varphi_1:\mathcal{D}_1\ra{\mathcal{D}_0}$ is defined 
componentwise as  $\varphi_1\vert_{\mathcal{D}_{1,\alpha}}=\id_{\Ho_{-1}(\{\varnothing\},\Bbbk)}$.

For $i\ge 2$, the maps $\varphi_i$ are defined using  
the Mayer-Vietoris sequence in reduced homology for the triple 
\begin{equation}\label{MVseq}
  \left(
  \D_\lambda,\hspace{.1in}
  \bigcup_{\stackrel{\beta\ld\alpha}{\lambda\ne\beta}}\D_\beta,\hspace{.1in}
  \Delta_\alpha
  \right)
\end{equation}
where $\D_\lambda=\Delta(\hat{0},\lambda]$ for all $\lambda\ld\alpha$.
For notational simplicity, when $\lambda\ld\alpha$ let $$\Delta_{\alpha,\lambda}=
\D_{\lambda}\cap\left(\bigcup_{\stackrel{\beta\ld\alpha}{\lambda\ne\beta}}
\D_\beta\right)$$
so that 
$\iota:\Ho_{i-3}(\Delta_{\alpha,\lambda},\Bbbk)\ra\Ho_{i-3}(\Delta_\lambda,\Bbbk)$ 
is the map induced in homology by the inclusion map and
$$
\delta_{i-2}^{\alpha,\lambda}:
\Ho_{i-2}(\Delta_\alpha,\Bbbk)\ra
\Ho_{i-3}(\Delta_{\alpha,\lambda},\Bbbk)
$$ 
is the connecting homomorphism from the Mayer-Vietoris sequence 
in homology of (\ref{MVseq}).  
For $i\ge 2$ the map $\varphi_i:\mathcal{D}_i\ra \mathcal{D}_{i-1}$ is defined componentwise by   $$\varphi_{i}\vert_{\mathcal{D}_{i,\alpha}}=\sum_{\lambda\ld\alpha}\varphi_{i}^{\alpha,\lambda}$$
where $$\varphi_{i}^{\alpha,\lambda}:\mathcal{D}_{i,\alpha}\ra{\mathcal{D}_{{i-1},\lambda}}$$ 
is the composition $\varphi_{i}^{\alpha,\lambda}=\iota\circ\delta_{i-2}^{\alpha,\lambda}$.

We now describe the process by which the sequence of vector spaces 
$\mathcal{D}(P)$ is transformed into a sequence of multigraded modules.  
For a monomial $m=x_1^{\bfa_1}\cdots x_d^{\bfa_d}\in R$ 
we write $\mdeg(m)=(\bfa_1,\ldots,\bfa_d)$ and 
$\deg_{x_\ell}(m)=\bfa_\ell$ for $1\le\ell\le d$.
Assuming the existence of a map of partially ordered sets 
$\eta:P\lra\mbb{N}^n$, the sequence of vector spaces 
$\mathcal{D}(P)$ is \emph{homogenized} to produce 
$$
\mathcal{F}(\eta):
  \cdots
  \lra 
  \mathcal{F}_t
  \stackrel{\partial_t}{\lra} 
  \mathcal{F}_{t-1}
  \lra
  \cdots
  \lra
  \mathcal{F}_1
  \stackrel{\partial_1}{\lra} 
  \mathcal{F}_0,
$$
a sequence of free multigraded 
$R$-modules and multigraded $R$-module homomorphisms. 
Indeed, homogenization of $\mathcal{D}(P)$ is carried out by constructing 
$F_0=R\otimes_{\Bbbk}\mathcal{D}_0$ and multigrading the result 
with $\mdeg(x^\bfa\otimes{v})=\bfa$ for each $v\in \mathcal{D}_0$.  
Similarly, for $i\ge{1}$, we set 
$$
\ds \mathcal{F}_i=
\bigoplus_{{\hat{0}}\ne\lambda\in P}\mathcal{F}_{i,\lambda}=
\bigoplus_{{\hat{0}}\ne\lambda\in P}R\otimes_{\Bbbk}\mathcal{D}_{i,\lambda}
$$ 
where the grading is defined as 
$\mdeg(x^\bfa\otimes{v})=\bfa+\eta(\lambda)$ for each 
$v\in \mathcal{D}_{i,\lambda}$.  The differential in this sequence 
of multigraded modules is defined 
componentwise in homological degree 1 as  
$$
\partial_1\arrowvert_{F_{1,\lambda}}
=x^{\eta(\lambda)}\otimes\varphi_1\arrowvert_{\mathcal{D}_{1,\lambda}}
$$ 
and for $i\ge{1}$, the map 
$\partial_i:\mathcal{F}_i\lra \mathcal{F}_{i-1}$ is 
defined as 
$$
\ds \partial_i\arrowvert_{\mathcal{F}_{i,\alpha}}
=\sum_{\lambda\ld\alpha}\partial^{\alpha,\lambda}_i
$$ 
where 
$\partial^{\alpha,\lambda}_i:\mathcal{F}_{i,\alpha}\lra{\mathcal{F}_{i-1,\lambda}}$ 
takes the form 
$\partial^{\alpha,\lambda}_i=x^{\eta(\alpha)-\eta(\lambda)}\otimes\varphi_{i}^{\alpha,\lambda}$ 
for $\lambda\ld\alpha$.  
The sequence $\mathcal{F}(\eta)$ approximates a free 
resolution of the multigraded module $R/M$ where $M$ 
is the ideal in $R$ generated by the monomials 
$$\{x^{\eta(a)}:a\in{A}\}$$ whose multidegrees 
are given by the images of the atoms of $P$.  

\begin{definition}\cite{Clark}
If\/ $\mathcal{F}(\eta)$ is an acyclic complex of 
multigraded modules, then we say that it is 
a \emph{poset resolution} of the ideal $M$.   
\end{definition}

Throughout the remainder of the paper, $N$ will denote 
a stable monomial ideal in $R$ and we write $G(N)$ 
as the unique minimal generating set of $N$.
For a monomial $m\in R$ set 
$$\max(m)=\max\{k\mid x_k\textrm{ divides }m\}$$ 
and $$\min(m)=\min\{k\mid x_k\textrm{ divides }m\}.$$  
To describe further the class of stable ideals, 
let $[d-1]=\{1,\ldots,d-1\}$, for 
$I\subseteq[d-1]$ let $\max(I)=\max\{i\mid{i}\in{I}\}$ 
and write $x_I=\prod_{i\in I}x_i$.

In Lemma 1.2 of \cite{EK}, Eliahou and Kervaire prove that a monomial 
ideal $N$ is stable if and only if for each monomial $m\in N$ 
there exists a unique $n\in{G(N)}$ with the property 
that $m=n\cdot{y}$ and $\max(n)\le\min(y)$. 
We adopt the language and notation introduced in the paper of 
Eliahou and Kervaire, and refer to $n$ as the \emph{unique decomposition} 
of the monomial $m$.  Following their convention, we encode 
this property in a \emph{decomposition map} \mbox{$\g:M(N)\lra{G(N)}$} 
where $M(N)$ is the collection of monomials of $N$ and $\g(m)=n$.  

\begin{definition}\textup{\cite{EK}}
An \emph{admissible symbol} is an ordered pair  
$(I,m)$ which satisfies $\max(I)<\max(m)$, where 
$m\in{G(N)}$ and $I\subseteq[d-1]$.
\end{definition}

\begin{definition}
The \emph{poset of admissible symbols} is the set 
$P_N$ of all admissible symbols associated to $N$, 
along with the symbol $\hat{0}=(\varnothing,1)$ which 
is defined to be the minimum element of $P_N$.
The partial ordering on $P_N$ is 
\begin{eqnarray*}
(J,n)\le{(I,m)} & \iff & J\subseteq{I} \textup{ and there exists } \\
& & C\subseteq{I\setminus{J}} \textup{ so that } n=\g(x_{C}m)
\end{eqnarray*}
when both symbols are admissible. 
\end{definition}

In the case when $(J,n)<{(I,m)}$ and 
$I=J\cup\{\ell\}$ for some $\ell$, 
then we write $(J,n)\ld{(I,m)}$ 
to describe the covering that occurs in $P_N$.  
As constructed, we have 
$\hat{0}\ld{\vx{m}}$ for every $m\in{G(N)}$. 
We are now in a position to state our first main result.

\begin{theorem}\label{stableTheorem}
Suppose that $N$ is a stable monomial ideal with 
poset of admissible symbols $P_N$ and define the map 
$\eta:P_N\lra\mbb{N}^n$ so that $(I,m)\mapsto\mdeg(x_I m)$.  
Then the complex $\mathcal{F}(\eta)$ is a minimal poset resolution of $R/N$. 
\end{theorem}

In order to prove Theorem \ref{stableTheorem}, we first 
describe the combinatorial structure of $P_N$ and then exhibit 
the connection between the complex $\mathcal{F}(\eta)$ and 
the minimal free resolution of the stable ideal $N$ 
constructed by Eliahou and Kervaire in \cite{EK}.  

\mysection{The Shellability of $P_N$}\label{PNShellable}
We begin this section by recalling some general facts regarding 
the shellability of partially ordered sets.  
Recall that a poset $P$ is called \emph{shellable} if the 
facets of its order complex $\Delta(P)$ can be arranged in a 
linear order $F_1,F_2,\ldots,F_t$ in such a way that 
the subcomplex 
$$\ds\left(\bigcup_{i=1}^{k-1}F_i\right)\bigcap{F_k}$$ 
is a nonempty union of maximal proper faces of $F_k$ for 
$k=2,\ldots,t$.  Such an ordering of facets is called a 
\emph{shelling}.

\begin{definition} 
Let $\mathcal{E}(P)$ denote the collection of edges in the
Hasse diagram of a poset $P$.  An \emph{edge labeling} of $P$ is a map 
$\lambda:\mathcal{E}(P)\lra\Lambda$ where $\Lambda$ is some poset.
\end{definition}

For $\sigma=a_1\ld\cdots\ld a_k$, a maximal chain of $P$, 
the \emph{edge label} of $\sigma$ is the sequence of labels 
$\lambda(\sigma)=\left(\lambda(a_1\ld a_2),\ldots,\lambda(a_{k-1}\ld a_k)\right)$. 

\begin{definition}
An edge labeling $\lambda$ is called an \emph{EL-labeling 
(edge lexicographical labeling)} if for every interval $[x,y]$ in $P$,
\begin{enumerate}
   \item There is a unique maximal chain $\sigma$ in $[x,y]$, 
   such that the labels of $\sigma$ form an increasing sequence 
   in $\Lambda$.  We call $\sigma$ the unique increasing maximal chain in $[x,y]$.
   \item $\lambda(\sigma)<\lambda(\sigma')$ under the lexicographic partial
    ordering in $\Lambda$ for all other maximal chains $\sigma'$ in $[x,y]$.
\end{enumerate}
A graded poset that admits an EL-labeling is said to be 
\emph{EL-shellable (edge lexicographically shellable)}.
\end{definition}

We further recall the following fundamental result of Bj\"orner and Wachs.  

\begin{theorem}\cite{BjWachs} 
EL-shellable posets are shellable.
\end{theorem}

We now define an edge labeling of the poset of admissible 
symbols $P_N$.
 
\begin{definition}\label{ELShelling}
Let $\lambda:P_N\ra\mbb{Z}$ take the form
\begin{displaymath}
\lambda\left((J,n)\ld(I,m)\right)=
\left\{
\begin{array}{ll}
0 & \textrm{if $n=1$}\\
-\ell & \textrm{if $n=m$}\\
\ell & \textrm{if $n\ne m$,}
\end{array} 
\right.
\end{displaymath}
where $\{\ell\}=I\smsm J$. 
\end{definition}

\begin{example}\label{PosetExample}
The labeled Hasse diagram for the poset of admissible symbols, $P_N$, 
of the stable ideal $N=\langle a,b,c\rangle^2=\langle a^2,ab,ac,b^2,bc,c^2 \rangle$ 
is 
\centering
\includegraphics[scale=0.4]{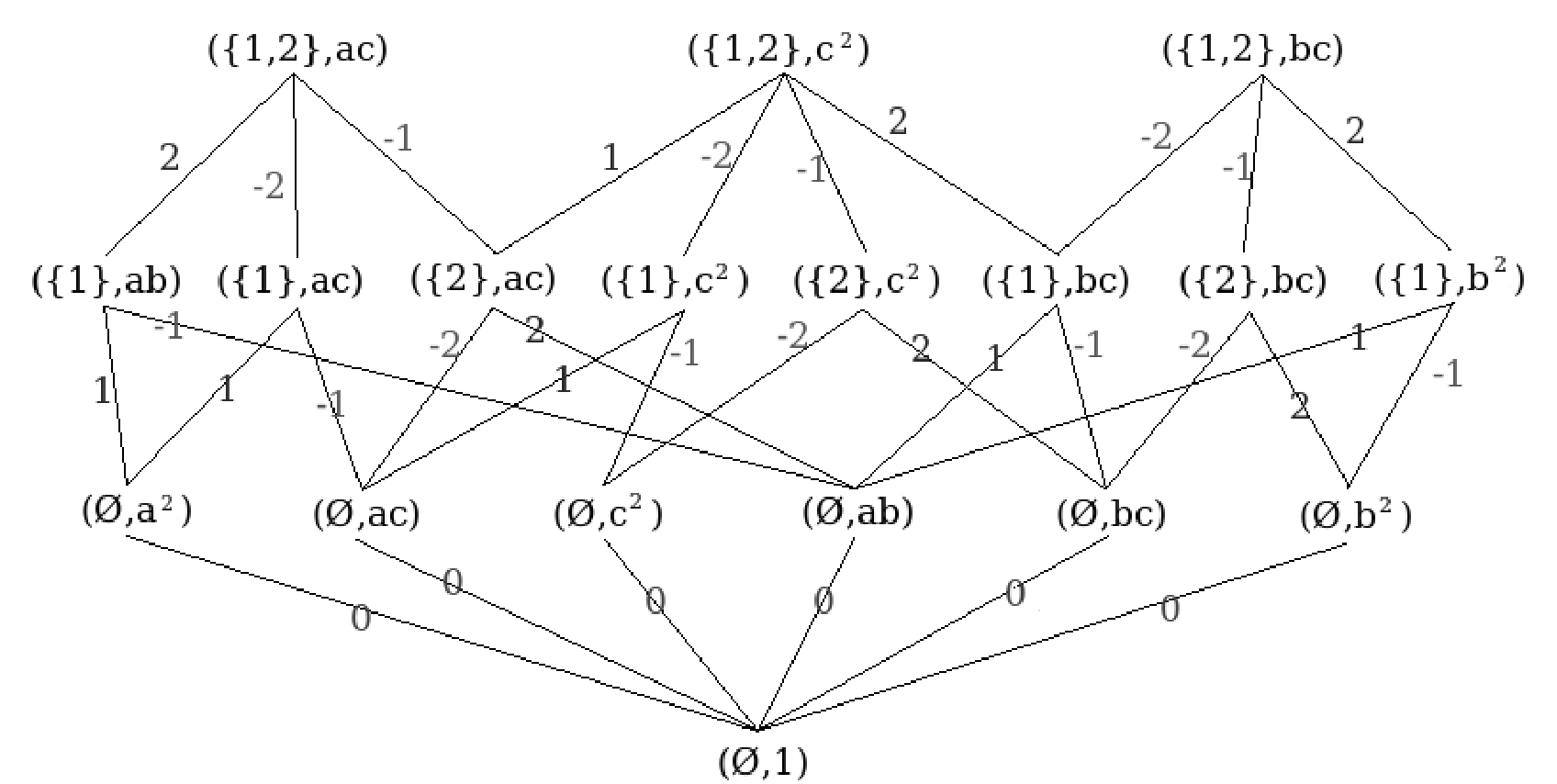}
\end{example}

Recall that given a poset $P$, the \emph{dual poset} 
$P^*$ has an underlying set identical to that of $P$, 
with $x<y$ in $P^*$ if and only if $y<x$ in $P$.  
Further, an edge labeling of a poset $P$ may also be 
viewed as an edge labeling of its dual poset and 
$P$ is said to be \emph{dual shellable} if $P^*$ is 
a shellable poset.  

\begin{theorem}\label{PNshellable}
The poset $P_N$ is dual EL-shellable 
with $\lambda$ defined as above.
\end{theorem}

Before turning to the proof of Theorem \ref{PNshellable}, 
we discuss some properties of the decomposition map $\g$ and 
the edge labeling $\lambda$.

\begin{remarks}\label{StableRemarks}$\phantom{glub}$
\begin{enumerate}
\item \cite[Lemma 1.3]{EK} For any monomial $w$ and 
      any monomial $m\in N$, we have $\g(w\g(m))=\g(wm)$ 
      and $\max(\g(wm))\le\max(\g(m))$.  We refer to 
      the first property as the \emph{associativity} of $\g$.
\item Suppose that $\left[(I,m),(J,n)\right]$ is a closed 
      interval in the dual poset $P^*_N$.
      Given a sequence of labels 
      $$\left(l_1,\ldots,l_k\right)$$ 
      there is at most one maximal chain $\sigma$ 
      in the closed interval such that 
      $$\lambda(\sigma)=\left(l_1,\ldots,l_k\right).$$
      When it exists, this chain must be equal to
      $$
        (I,m)\ld
        (I\smsm\{\ell_1\},n_1)\ld
        \cdots\ld
        (I\smsm\{\ell_1,\ldots,\ell_{k-1}\},n_{k-1})\ld
        (J,n)
      $$
      where $\ell_i=|l_i|$, the set $I\smsm J=\{\ell_1,\ldots,\ell_k\}$ and  
           \begin{displaymath}
           n_i=\left\{ 
           \begin{array}{cl}
           \g(x_{\ell_i}n_{i-1}) & \textrm{if } l_i>0 \\
           n_{i-1} & \textrm{if } l_i<0
           \end{array}
           \right.
           \end{displaymath}
      for $1\le i\le k$ with $n_0=m$ and $n_k=n$.
\end{enumerate}
\end{remarks}

Suppose that $(J,n)<(I,m)$ is a pair of comparable admissible symbols.
      Then $n=\g(x_{C'}m)$ for some $C'\subseteq I\smsm J$.  
      Let $C=\{c\in C'\mid c\le\max(n)\}$.  Then by 
      associativity and \cite[Lemma 1.2]{EK} we have 
      $n=\g(x_{C'}m)=\g(x_{C'\smsm C}\g(x_{C}m))=\g(x_{C}m)$.  
      In this way, any representation of $n=\g(x_{C'}m)$ may 
      be reduced to $n=\g(x_{C}m)$ under the conditions above. 
\begin{notation}
      Implicit in all subsequent arguments is the convention 
      that a representation $n=\g(x_{C}m)$ is written in reduced form.  
\end{notation}

\begin{lemma}\label{UniqueRep}
For a reduced representation of $n=\g(x_{C}m)$
the set $C$ is the unique subset of minimum cardinality 
among all $C'\subseteq I\smsm J$ for which $n=\g(x_{C'}m)$. 
\end{lemma}       

\begin{proof}
Suppose that $C$ is not the subset of $I\smsm J$ with 
      smallest cardinality, namely that there exists $D\subseteq I\smsm J$ 
      with $|D|<|C|$ and $n=\g(x_Dm)$.  By definition, $x_C\cdot m=n\cdot y$ 
      and $x_D\cdot m=n\cdot u$ where $\max(n)\le\min(y)$ and $\max(n)\le\min(u)$.  
      The assumption of $|D|<|C|$ implies that there exists $c\in C$ such that 
      $c\notin D$.  Rearranging and combining the two equations above, we arrive at the 
      equality $x_C\cdot u=x_D\cdot y$.  This equality allows us to conclude 
      that $x_c$ divides $y$ since it cannot divide $x_D$.  By definition, 
      $\max(n)\le\min(y)$ and therefore $\max(n)\le c$.  Further, since 
      we assumed that $n=\g(x_Cm)$ possessed the property that $c\le\max(n)$
      we have $c\le\max(n)\le c$ so that $\max(n)=c$.  This equality 
      also has implications for $x_D$ and $u$, namely that $c\le\min(u)$ 
      and $\max(D)\le\max(n)=c$ so that $\max(D)<c$ since $c\notin D$.  
      However, $c<\max(D)\le\max(n)=c$ is a contradiction, and our original 
      supposition that such a $D$ exists is false.  If $C$ and $D$ are distinct
      subsets of $I\smsm J$ with $|C|=|D|$ and $n=\g(x_Cm)=\g(x_Dm)$ 
      then there is a $c\in C$ and $d\in D$ for which $c\notin D$ and $d\notin C$.
      As before, we use the equality $x_C\cdot u=x_D\cdot y$ and now conclude 
      that $x_c$ divides $y$ and $x_d$ divides $u$.  Therefore, 
      $c\le\max(C)\le\max(n)\le\min(y)\le c$ and similarly 
      $d\le\max(D)\le\max(n)\le\min(u)\le d$ so that $\max(n)=c=d$, 
\end{proof} 

\begin{proof}[Proof of Theorem \ref{PNshellable}]$\phantom{2}$

To prove the dual EL-Shellability of $P_N$, 
recall that for the poset of admissible symbols $P_N$, 
we have comparability in the dual poset given by 
$(I,m)<(J,n)\in{P_N^*}$ if and only if $(J,n)<(I,m)\in{P_N}$. 
We proceed with the proof by considering the various types of 
closed intervals that appear in the dual poset $P_N^*$.  

\emph{Case 1:} Consider the closed interval 
$\left[(I,m),\hat{0}\right]$.  Write $I=\{d_1,\ldots,d_t\}$ so that 
$d_j<d_{j+1}$, for every $j=1,\ldots,t$.  
The maximal chain $$\sigma=(I,m)\ld
        (I\smsm\{d_t\},m)\ld
        (I\smsm\{d_t,d_{t-1}\},m)\ld
        \cdots\ld
        (\varnothing,m)\ld
        \hat{0}$$ 
has the increasing label 
$$\left(-d_{t},-d_{t-1},\ldots,-d_1,0\right).$$ 

Consider a maximal chain $\tau\in\left[(I,m),\hat{0}\right]$ where $\tau\ne\sigma$.  
If each label in the sequence $\lambda(\tau)$ (except the label of 
coverings of the form $(\varnothing,n)\ld\hat{0}$) is negative, 
then the sequence $\lambda(\tau)$ cannot be 
increasing, for it must be a permutation of the 
sequence $\lambda(\sigma)$ where the rightmost label $0$ is fixed.
If any label within the sequence $\lambda(\tau)$ is positive, 
then again $\lambda(\tau)$ cannot be increasing since 
every maximal chain contains the labeled subchain 
$$(\varnothing,n)\stackrel{0}{\ld}\hat{0}$$ 
for every $(I,m)<(\varnothing,n)$. 
Therefore, $\sigma$ is the unique rising chain in 
the interval $\left[(I,m),\hat{0}\right]$.  Further, 
$\lambda(\sigma)$ is lexicographically first among all 
chains in $\left[(I,m),\hat{0}\right]$ since $-d_t<\cdots<-d_1<0$.

\emph{Case 2:} 
Consider the closed interval $\left[(I,m),(J,n)\right]$ 
of $P_N^*$ where $(J,n)\ne\hat{0}$ and $n=m$.  Again write 
$I\smsm J=\{d_1,\ldots,d_t\}$ such that $d_1<\cdots<d_t$.  
Every maximal chain $\sigma$ in $\left[(I,m),(J,m)\right]$ has a label 
of the form $$\left(-d_{\rho(t)},\ldots,-d_{\rho(1)}\right)$$ where 
$\rho\in\Sy_t$ is a permutation of the set $\{1,\ldots,t\}$.  
Therefore, the label $$\left(-d_t,\ldots,-d_1\right)$$ 
corresponding to the identity permutation is the unique 
increasing label in $\left[(I,m),(J,m)\right]$ and is 
lexicographically first among all such labels.

\emph{Case 3:} 
Consider the closed interval $\left[(I,m),(J,n)\right]$ 
of $P_N^*$ where $(J,n)\ne\hat{0}$ and $m\ne n$.  
By Lemma \ref{UniqueRep}, $n=\g(x_C m)$ for 
a unique $C\subseteq I\smsm J$ where $\max(C)\le\max(n)$ and 
the set $C$ is of minimum cardinality.  
Writing the set $C=\{c_1,\ldots,c_q\}$ 
and $(I\smsm J)\smsm C=\{\ell_1,\ldots,\ell_t\}$ where 
$\ell_1<\cdots<\ell_t$ and $c_1<\cdots<c_q$, it follows 
that the sequence of edge labels 
$\left(-\ell_t,\ldots,-\ell_1,c_1,\ldots,c_q\right)$ 
is the increasing label of a maximal chain $\sigma$ in 
$\left[(I,m),(J,n)\right]$.  

Turning to uniqueness, suppose that 
$\tau\ne\sigma$ is also a chain which has a 
rising edge label.  Then 
\begin{equation}\label{uniquelabelcontradiction}
\lambda(\tau)=\left(-d_p,\ldots,-d_1,s_1,\ldots,s_j\right)
\end{equation}
where $$\{s_1,\ldots,s_j\}\cup\{d_1,\ldots,d_p\}=\{c_1,\ldots,c_q\}\cup\{\ell_1,\ldots,\ell_t\}=I\smsm J,$$ 
and $-d_p<\ldots<-d_1<0<s_1<\ldots<s_j$.  
Since $\tau\ne\sigma$, then $\lambda(\tau)\ne\lambda(\sigma)$ and 
in particular, $\{d_1,\ldots,d_p\}\ne\{\ell_1,\ldots,\ell_t\}$.

If there exists $\ell\in\{\ell_1,\ldots,\ell_t\}$ 
with the property that $\ell\notin\{d_1,\ldots,d_p\}$, we must have  
$\ell\in\{s_1,\ldots,s_j\}$ so that $\ell=s_i$ for some $i<j$ and 
the label $\lambda(\sigma)$ has the form 
\begin{equation}
\left(-d_p,\ldots,-d_1,s_1,\ldots,\ell,\ldots,s_j\right).
\end{equation}
By the definition of $\g$, we have the equalities 
$x_C\cdot m=n\cdot y$ and $x_S\cdot m=n\cdot u$, which 
may be combined and simplified to arrive at the equation 
$x_C\cdot u=x_S\cdot y$.  The assumption that $\ell\in S$ 
and $\ell\notin C$ implies that $x_\ell$ divides $u$ so 
that $\max(n)\le\ell$.  It therefore follows that 
$\max(n)\le\ell<s_{i+1}<\cdots<s_j$ when $\ell\ne s_j$ so that 
$\max(I\smsm\{s_1,\ldots,\ell\})=s_j>\max(n)$, which 
is a contradiction the admissibility of the symbol 
$(I\smsm\{d_1,\ldots,d_p,s_1,\ldots,\ell\},n)$.  
If $\ell=s_j$ then $\max(n)\le s_j$ and therefore 
$n=\g(x_{s_1}\cdots x_{j} m)=\g(x_{s_1}\cdots x_{j-1} m)$ 
which implies that the symbol 
$(I\smsm\{s_1,\ldots,s_{j-1}\},\g(x_{s_1}\cdots x_{j-1} m))$ 
which would neccesarily precede $(J,n)$ in the chain is not admissible.  
If there exists $d_g\in \{d_1,\ldots,d_p\}$ 
with $d_g\notin\{\ell_1,\ldots,\ell_t\}$ then 
a similar argument again provides a contradiction 
to admissibility.  

We now prove that $\lambda(\sigma)$ is lexicographically 
smallest among all chains.  Aiming for a contradiction, 
suppose that the label $\lambda(\sigma)$ is not lexicographically 
smallest so that there exists a maximal dual chain $\tau$ with 
$\lambda(\tau)<\lambda(\sigma)$.  Without loss of generality, we 
may assume that $\lambda(\sigma)$ and $\lambda(\tau)$ differ 
at their leftmost label $-c$, where $-c<-\ell_t$.  
Such a $c$ must be an element of 
the set $C$ since $-\ell_t<\cdots<-\ell_1$ is inherent in the 
structure of $\lambda(\sigma)$.  
By construction, $c\in C$ implies that $c\le\max(n)$ and 
utilizing the equations $x_C\cdot m=n\cdot y$ and $x_S\cdot m=n\cdot u$, 
to produce $x_C\cdot u=x_S\cdot y$, it follows that $x_c$ divides $y$ 
and therefore $c\le\max(n)\le\min(y)\le c$ so that $\max(n)=c$.  
This forces the element $c=c_q$ for otherwise, the chain $\sigma$ 
would contain the subchain  
$(I\smsm\{\ell_1,\ldots,\ell_t,c_1,\ldots,c\},n)<(I\smsm J,n)$ 
where $(I\smsm\{\ell_1,\ldots,\ell_t,c_1,\ldots,c\},n)$ is 
not an admissible symbol.  

The desired contradiction will be obtained 
within an investigation of each of the three 
possibilities for the relationship 
between $\deg_{x_c}(n)$ and $\deg_{x_c}(m)$.

Suppose $\deg_{x_c}(n)>\deg_{x_c}(m)$ so that 
$\deg_{x_c}(n)=\deg_{x_c}(m)+1$, based upon the structure of the 
set $I$ and the definition of the decomposition map $\g$.  
In this case, the chain $\tau$ cannot end in $(J,n)$ since $-c$, 
the leftmost label of $\tau$, labels the subchain $(I,m)\ld(I\smsm\{c\},m)$ 
and the $x_c$ degree of every monomial appearing in the chain $\tau$ may not increase.  

If $\deg_{x_c}(n)<\deg_{x_c}(m)$ then the unique decomposition 
$x_C\cdot m=n\cdot u$ implies that $x_c$ divides $u$ , for otherwise 
$c\in C$ implies that $\deg_{x_c}(n)=\deg_{x_c}(m)+1$, a contradiction.  
The conclusion that $x_c$ divides $u$ allows $x_C\cdot m=n\cdot u$ 
to be simplified to $x_{C'}\cdot m=n\cdot u'$ where $C'=C\smsm\{c\}$ 
and $u'=u/x_c$.  This contradicts the condition that $C$ is the set of 
smallest cardinality for which $\g(x_Cm)=n$.  

Lastly, if $\deg_{x_c}(n)=\deg_{x_c}(m)$ we turn to the chain $\sigma$, 
whose rightmost label is $c$.  The subchain with this label is 
$(I\smsm\{\ell_1,\ldots,\ell_t,c_1,\ldots,c_{j-1}\},n')<(I\smsm J,n)$ 
where $x_c\cdot n'=n\cdot y$ where $n$ does not contain this new 
factor of $x_c$.  The monomial $x_c$ therefore divides $y$ and  
we can reduce $x_c\cdot n'=n\cdot y$ to $n'=n\cdot u'$ where $u'=u/x_c$, 
a contradiction to $n'\in G(N)$.  This completes the proof.
\end{proof}

With Theorem \ref{PNshellable} established, 
we immediately have the following corollary.

\begin{corollary}\label{finshell}
Every interval of $P_N$ which is of the form 
$\left[\hat{0},(I,m)\right]$ is finite, dual 
EL-shellable and therefore shellable.
\end{corollary}

\mysection{The topology of $P_N$ and properties of $\mathcal{D}(P_N)$}\label{PropertiesOfCPN}
To establish the connection between the poset $P_N$ and 
the sequence $\mathcal{D}(P_N)$, we recall the 
definition of \emph{CW poset}, due to Bj\"orner \cite{BjCW}.

\begin{definition}\cite{BjCW}
A poset $P$ is called a \emph{CW poset} if
\begin{enumerate}
\item $P$ has a least element $\hat{0}$,
\item $P$ is nontrivial (has more than one element),
\item For all $x\in{P}\smsm\{\hat{0}\}$, the open 
      interval $(\hat{0},x)$ is homeomorphic to a sphere.
\end{enumerate}
\end{definition}

After establishing this definition, Bj\"orner 
describes sufficient conditions for  
a poset to be a CW poset.  

\begin{proposition}
\cite[Proposition 2.2]{BjCW}
Suppose that $P$ is a nontrivial poset such that
\begin{enumerate}
\item $P$ has a least element $\hat{0}$,
\item every interval $[x,y]$ of length two has cardinality four,
\item For every $x\in{P}$ the interval $[\hat{0},x]$ is finite and shellable.
\end{enumerate}
Then $P$ is a CW poset.
\end{proposition}

With this proposition in hand, we now may conclude 
the following about the structure of $P_N$, 
the poset of admissible symbols.  

\begin{theorem}\label{PNCWposet}
The poset of admissible symbols $P_N$ is a CW poset.
\end{theorem}

\begin{proof}
The poset $P_N$ has a least element by construction and each of its 
intervals $\left[\hat{0},(I,m)\right]$ is finite and shellable by 
Corollary \ref{finshell}.  Thus, it remains to show that every 
closed interval in $P_N$ of length two has cardinality four.  

\emph{Case 1:} 
Let $(J,n)=\hat{0}$ so that the set $I$ is a singleton.   
It follows that the only poset elements in the open interior of the 
interval are $(\varnothing,m)$ and $(\varnothing,\g(x_I m))$. 

\emph{Case 2:}
Let $(J,n)\ne\hat{0}$ and suppose that 
$\left[(J,n),(I,m)\right]$ is a closed interval of length 
two in the poset of admissible symbols, $P_N$.  
Since the interval is of length two, the set $J$ 
has the form $I\smsm\{i_0,i_1\}$ for some $i_0<i_1\in{I}$.  
Further, any poset element in the interval must have 
either $I\smsm\{i_0\}$ or $I\smsm\{i_1\}$ as its first 
coordinate, for these sets are the only subsets of $I$ which 
contain $I\smsm\{i_0,i_1\}$.  

Write $m=m'x_{i_2}x_{i_3}$ where 
$\max(m')\le i_2\le i_3$.  We must now consider 
each of the possible orderings for the elements 
of the (multi) set $\{i_0,i_1,i_2,i_3\}$ to ascertain 
the choices available for the monomial $n$.  Our assumptions of the 
inequalities $i_0<i_1$ and $i_2\le i_3$ together with 
the admissibility of the symbol $(I,m)$ imply that 
$i_1\le\max(I)<\max(m)\le i_3$.  Hence, determining 
the number of orderings amounts to producing 
a count of the number of orderings for elements of 
the set $\{i_0,i_1,i_2\}$, of which there are three, 
since $i_0<i_1$.  

\emph{Subcase 2.1:}
Suppose that $i_0<i_1<i_2\le i_3$.  

If $n=m$, then the poset elements which are contained in the 
open interior of the interval are forced to be 
$(I\smsm\{i_0\},m)$ and $(I\smsm\{i_1\},m)$.  

If $n=\g(x_{i_0}m)$ and $\max(I\smsm\{i_0\})<\max(\g(x_{i_0}m))$ 
then the symbol $(I\smsm\{i_0\},\g(x_{i_0}m))$ is admissible, 
so that it is in the open interior of the 
interval along with the admissible symbol $(I\smsm\{i_1\},m)$.  
The symbol $(I\smsm\{i_0\},m)$ is not comparable to 
$(I\smsm\{i_0\},\g(x_{i_0}m))$ due to the absence of the value 
$i_0$.  The symbol $(I\smsm\{i_1\},\g(x_{i_1}m))$ is also 
not comparable to $(I\smsm\{i_0\},\g(x_{i_0}m))$, for if it were 
then either $\g(x_{i_0m})=\g(x_\varnothing\g(x_{i_1}m))=\g(x_{i_1}m)$ 
or $\g(x_{i_0}x_{i_1}m)=n=\g(x_{i_0}m)$.  The first equality is 
impossible since Lemma \ref{UniqueRep} guarantees that 
$\{i_0\}$ is the unique set containing one element 
for which $n=\g(x_{i_0}m)$.  The second equality also can not occur 
since Lemma 1.2 of \cite{EK} guarantees monomial equality 
$\g(x_{i_0}x_{i_1}m)=\g(x_{i_0}m)$ if and only if 
$\max(n)\le\min(x_{i_1})=i_1$, which would contradict the assumption 
that $(I\smsm\{i_0\},n)$ is an admissible symbol. 

If $n=\g(x_{i_0}m)$ and $\max(I\smsm\{i_0\})\ge\max(\g(x_{i_0}m))$ 
then the symbol $(I\smsm\{i_0\},\g(x_{i_0}m))$ is not admissible and 
is not an element of $P_N$.  
However, we are assuming that the symbol 
$(I\smsm\{i_0,i_1\},\g(x_{i_0}m))$ is admissible, so that 
$\max(\g(x_{i_0}m))\le i_1$ and via Lemma 1.2 of \cite{EK}, 
we have the monomial equality $\g(x_{i_0}\g(x_{i_1}m))=\g(x_{i_1}m)$.  
Therefore, $(I\smsm\{i_0\},\g(x_{i_0}m))=(I\smsm\{i_0\},\g(x_{i_0}x_{i_1}m))$ 
and the symbols $(I\smsm\{i_1\},m)$ and $(I\smsm\{i_1\},\g(x_{i_1}m))$ 
are each contained in the interval.  Since $n=\g(x_{i_0}m)$, 
the symbol $(I\smsm\{i_0\},m)$ is not comparable to 
$(I\smsm\{i_0\},\g(x_{i_0}m))$.  

If $n=\g(x_{i_1}m)$ then the symbol 
$(I\smsm\{i_0\},m)$ is certainly contained in 
the closed interval.  Further, $(I\smsm\{i_1\},\g(x_{i_1}m))$ 
must be admissible for were it not, then the assumption 
of admissibility for $(I\smsm\{i_0,i_1\},\g(x_{i_1}m))$ 
implies that $i_0\ge\max(\g(x_{i_1}m))\ge\min(\g(x_{i_1}m))\ge i_1$, 
a contradiction to the initial stipulation that $i_0<i_1$.  
The symbol $(I\smsm\{i_1\},m)$ is incomparable to 
$(I\smsm\{i_0,i_1\},\g(x_{i_1}m))$ and were 
$(I\smsm\{i_0\},\g(x_{i_0}m))$ comparable to 
$(I\smsm\{i_0,i_1\},\g(x_{i_1}m))$, then either 
$\g(x_{i_1m})=\g(x_\varnothing\g(x_{i_0}m))=\g(x_{i_0m})$ 
or $\g(x_{i_0}x_{i_1}m)=n=\g(x_{i_0}m)$.  The first 
equality contradicts Lemma \ref{UniqueRep} and 
the second may be used to arrive at a contradiction to 
the admissibility of $(I\smsm\{i_1\},\g(x_{i_1}m))$.  
These arguments are similar to those used in when 
$n=\g(x_{i_0}m)$ and $\max(I\smsm\{i_0\})<\max(\g(x_{i_0}m))$.  

If $n=\g(x_{i_0}x_{i_1}m)$ and $n=\g(x_{i_0}x_{i_1}m)\ne\g(x_{i_0}m)$ 
then the symbols $(I\smsm\{i_0\},\g(x_{i_0}m))$ and 
$(I\smsm\{i_1\},\g(x_{i_1}m))$ are admissible and are contained
in the open interior of the interval.  Clearly, the symbols 
$(I\smsm\{i_0\},m)$ and $(I\smsm\{i_1\},m)$ are not comparable to 
$(I\smsm\{i_0,i_1\},\g(x_{i_0}x_{i_1}m))$ in this instance.  
If $n=\g(x_{i_0}x_{i_1}m)$ and $n=\g(x_{i_0}x_{i_1}m)=\g(x_{i_0}m)$, 
then we are reduced to an already resolved case.  

For each of these four choices of $n$, the interval has four elements.  

\emph{Subcase 2.2} 
We now consider the two remaining orderings 
$i_0<i_2\le i_1<i_3$ and $i_2\le i_0<i_1<i_3$.  
Under each of these orderings, 
we have $\deg_{x_{i_3}}(m)=1$ and in light of Lemma 
1.3 of \cite{EK} if $n\ne m$ we have $\max(n)<\max(m)=i_3$ 
and in turn that $\max(n)\le i_1$.  

If $n=m$, then the poset elements which are contained in the 
open interior of the interval are forced to be 
$(I\smsm\{i_0\},m)$ and $(I\smsm\{i_1\},m)$.  

If $n=\g(x_{i_0}m)$ then $\max(n)\le i_1$ implies that 
the symbol $(I\smsm\{i_0\},\g(x_{i_0}m))$ is not admissible and 
is not an element of $P_N$.  
However, we are assuming that the symbol 
$(I\smsm\{i_0,i_1\},\g(x_{i_0}m))$ is admissible, so that 
$\max(\g(x_{i_0}m))\le i_1$ and again using Lemma 1.2 of \cite{EK}, 
we have $\g(x_{i_0}\g(x_{i_1}m))=\g(x_{i_0}m)$.  
Therefore, $(I\smsm\{i_0\},\g(x_{i_0}m))=(I\smsm\{i_0\},\g(x_{i_0}x_{i_1}m))$ 
and the symbols $(I\smsm\{i_1\},m)$ and $(I\smsm\{i_1\},\g(x_{i_1}m))$ 
are each contained in the interval.  Since $n=\g(x_{i_0}m)$, 
the symbol $(I\smsm\{i_0\},m)$ is not comparable to 
$(I\smsm\{i_0\},\g(x_{i_0}m))$.  

If $n=\g(x_{i_1}m)$ then the symbol 
$(I\smsm\{i_0\},m)$ is certainly contained in 
the closed interval.  Further, $(I\smsm\{i_1\},\g(x_{i_1}m))$ 
must be admissible for were it not, then the assumption 
of admissibility for $(I\smsm\{i_0,i_1\},\g(x_{i_1}m))$ 
implies that $i_0\ge\max(\g(x_{i_1}m))\ge\min((\g(x_{i_1}m)))\ge i_1$, 
a contradiction to the initial stipulation that $i_0<i_1$.  
The symbol $(I\smsm\{i_1\},m)$ is incomparable to 
$(I\smsm\{i_0,i_1\},\g(x_{i_1}m))$ and were 
$(I\smsm\{i_0\},\g(x_{i_0}m))$ comparable to 
$(I\smsm\{i_0,i_1\},\g(x_{i_1}m))$, then either 
$\g(x_{i_1m})=\g(x_\varnothing\g(x_{i_0}m))=\g(x_{i_0m})$ 
or $\g(x_{i_0}x_{i_1}m)=n=\g(x_{i_0}m)$.  The first 
equality contradicts Lemma \ref{UniqueRep} and 
the second may be used to arrive at a contradiction to 
the admissibility of $(I\smsm\{i_1\},\g(x_{i_1}m))$.  
These arguments are similar to those used in the case 
when $n=\g(x_{i_0}m)$ and $\max(I\smsm\{i_0\})<\max(\g(x_{i_0}m))$.  

Again, for each of these three choices of $n$, the interval has four elements.  

\end{proof}

We now analyze the vector spaces which are present in the sequence 
$\mathcal{D}(P_N)$ at the level of individual poset elements.  
In order to do so, we recall the following combinatorial results.  
As is standard, we write $\bar{P}=P\smsm\{\hat{0},\hat{1}\}$.

\begin{theorem}[\cite{BjShell,BjWachsNP1}]
If a bounded poset $P$ is EL-shellable, then the lexicographic 
order of the maximal chains of $P$ is a shelling of $\Delta(P)$.  
Moreover, the corresponding order of the maximal chains of $\bar{P}$ 
is a shelling of $\Delta(\bar{P})$.
\end{theorem}

\begin{theorem}[\cite{BjWachsNP1}]\label{BjWachsELBasis}
Suppose that $P$ is a poset for which $\hat{P}=P\cup\{\hat{0},\hat{1}\}$ 
admits an EL-labeling.  Then $P$ has the homotopy 
type of a wedge of spheres.  Furthermore, for any fixed EL-labeling:
\begin{enumerate}
\item[i.] $\Ho_i(\Delta(P),\mbb{Z})\cong\mbb{Z}^{\#\textup{falling chains of length }i+2}$,
\item[ii.] bases for $i$-dimensional homology (and cohomology) are induced by the falling
chains of length $i+2$.
\end{enumerate}
\end{theorem}

In the analysis that follows, we again 
examine the dual poset $P_N^*$ and focus our 
attention on the collection of closed intervals 
of the form $[(I,m),\hat{0}]$, to each of which we 
apply Theorem \ref{BjWachsELBasis}.  
Indeed, for each admissible symbol $(I,m)\in P_N^*$ 
where $|I|=q$, the open interval $((I,m),\hat{0})$ 
is homeomorphic to a sphere of dimension $q-1$ 
since $P_N$ is a CW poset.  
Further, the EL-labeling of $[(I,m),\hat{0}]$ guarantees that 
the unique generator of $\Ho_{q-1}(\Delta_{I,m},\Bbbk)$ 
is induced by a unique falling chain of length $q+1$.  
In the discussion that follows, we use the EL-shelling of 
Definition \ref{ELShelling} to produce a canonical generator 
of $\Ho_{q-1}(\Delta_{I,m},\Bbbk)$ 
as a linear combination in which each facet of 
$\Delta_{I,m}$ occurs with coefficient $+1$ or $-1$.  

To begin, consider a maximal chain $(I,m)\ld\sigma\ld\hat{0}$ 
which is of length $q+1$ and appears in the dual 
closed interval $\left[(I,m),\hat{0}\right]$ and 
write the label of this chain as 
\begin{eqnarray}\label{labelseq}
\left(
l^\sigma_1,
\ldots,
l^\sigma_q,
0
\right).
\end{eqnarray}
We note that $I=\{|l^\sigma_1|,\ldots,|l^\sigma_q|\}$ and write 
\begin{eqnarray}\label{facetsigns}
\varepsilon_\sigma & = & \sgn(\rho_\sigma)\cdot\sgn\left(\prod_{t=1}^q l_q\right)
\end{eqnarray}
where $\rho_\sigma\in\Sy_q$ is the permutation arranging the sequence
$$
|l^\sigma_1|,\ldots,|l^\sigma_q|
$$
in \emph{increasing} order.  
We endow the corresponding chain $\sigma$ in 
$\left((I,m),\hat{0}\right)$ with this sign 
$\varepsilon_\sigma$ and refer to it as the \emph{sign}
of $\sigma$.

The unique maximal chain $\tau$ in $\left[(I,m),\hat{0}\right]$ 
which has a decreasing label is 
the chain consisting of admissible symbols 
having at each stage a different monomial as their second coordinate 
and the sequence of sets 
$$
I,I\smsm\{i_q\},I\smsm\{i_{q-1},i_q\},
\ldots,
\{i_1,i_2\},\{i_1\},\varnothing
$$
as their first coordinate.  The unique 
falling chain $\tau\in\left[(I,m),\hat{0}\right]$ is 
therefore 
$$
(I,m)\ld
(I_{q},m_q)\ld 
(I_{q-1,q},m_{q-1,q})\ld
\cdots 
\ld(I_{2,\ldots,q},m_{2,\ldots,q})\ld
(\varnothing,m_{1,\ldots,q})\ld
\hat{0}
$$
where $I=\{i_1,\ldots,i_q\}$ with $i_1<\ldots<i_q$ and 
for $j=1,\ldots, q$, the set $I_{j,\ldots,q}=I\smsm\{i_j,\ldots,i_q\}$ and 
the monomial $m_{j,\ldots,q}=\g(x_{i_j}\cdots x_{i_q}m)$.  
The label of the chain $\tau$ is therefore  
$$\left(i_q,\ldots,i_1,0\right)$$ and is decreasing.  If there 
were another such chain with decreasing label, 
then such a chain would be counted by Theorem 
\ref{BjWachsELBasis} and $\left(\hat{0},(I,m)\right)$ would 
not have the homotopy type of a sphere, a contradiction 
to the fact that $P_N$ is a CW poset.  In the context of 
the shelling order produced by the EL-shelling 
above, the chain $\tau$ appears lexicographically last 
among all maximal chains in the dual interval 
and is therefore the unique homology facet of 
$\Delta_{I,m}$.  

\begin{definition}
For an admissible symbol $(I,m)$, set 
$$
  f(I,m)
  =\ds\sum_{\sigma\in\left((I,m),\hat{0}\right)}
   \varepsilon_\sigma\cdot\sigma,
$$
the linear combination of all maximal chains of 
the open interval $\left((I,m),\hat{0}\right)$ with 
coefficients given by (\ref{facetsigns}).
\end{definition}

Viewing the maximal chains of $\left((I,m),\hat{0}\right)$ 
as facets in the order complex $\Delta_{I,m}$ 
we now establish the following.

\begin{lemma}\label{CycleLemma}
The sum $f(I,m)$ is a $q-1$-dimensional cycle in $\Ho_{q-1}(\Delta_{I,m},\Bbbk)$ 
which is not the boundary of any $q$-dimensional face. 
\end{lemma}

\begin{proof}
The maximal chains in the open interval 
$\left((I,m),\hat{0}\right)$ are each of length $q-1$, 
so that no $q$-dimensional faces are present in $\Delta_{I,m}$.  
Thus, $f(I,m)$ cannot be the boundary of a $q$-dimensional face 
of $\Delta_{I,m}$.  

We now show that $f(I,m)$ is a $q-1$-dimensional cycle.  
Suppose that $\sigma$ is a maximal chain in 
$\left((I,m),\hat{0}\right)$ 
and let $(J,n)$ be an element of said chain.  
We exhibit a unique chain $\sigma'$ which also appears in 
$f(I,m)$ and differs from $\sigma$ only at the element $(J,n)$.  

Indeed, consider the chain $(I,m)\ld\sigma\ld\hat{0}$ 
along with its subchain $(J_1,n_1)\ld(J,n)\ld(J_2,n_2)$.  
In the proof of Theorem \ref{PNCWposet}, each closed interval 
of length two was shown to be of cardinality four, and therefore 
there exists a unique $(J',n')\in[(J_1,n_1),(J_2,n_2)]$ which 
is not equal to $(J,n)$.  Defining $\sigma'$ by removing 
$(J,n)$ and replacing it with $(J',n')$, we have 
constructed the desired chain.  

We claim that for the chains $\sigma$ and $\sigma'$, 
the associated signs $\varepsilon_\sigma$ and 
$\varepsilon_{\sigma'}$ are opposite to one another.  

If $(J_2,n_2)=\hat{0}$ then $(J_1,n_1)=(\{j\},n_1)$ for some $j$.
Thus, $(J,n)=(\varnothing,n)$ and $(J',n')=(\varnothing,n')$ so that 
the chains $\sigma$ and $\sigma'$ have the same corresponding 
permutation $\rho$.  Since either $n_1=n$ or $n_1=n'$, 
without loss of generality we assume that $n_1=n$ 
so that $n'=\g(x_jn)$.  Therefore, the subchain 
$(\{j\},n)\ld(\varnothing,n)\ld\hat{0}$ 
has $-j$ as its label, 
while $(\{j\},n)\ld(\varnothing,n')\ld\hat{0}$ 
has $j$ as its label.  This is the only 
difference in the labels $\lambda(\sigma)$ and $\lambda(\sigma')$ 
and $\varepsilon_\sigma\ne\varepsilon_{\sigma'}$ is forced.

If $(J_2,n_2)\ne\hat{0}$ then for each case that appears in 
the classification of intervals of length two described in 
the proof of Theorem \ref{PNCWposet}, we can compute $\varepsilon_\sigma\ne\varepsilon_{\sigma'}$.

When the differential $d$ in the reduced chain complex 
$\CC(\Delta_{I,m})$ is applied to the sum $f(I,m)$, each term 
appears twice with opposite signs, so that $d(f(I,m))=0$ making 
$f(I,m)$ a $q-1$-dimensional cycle in $\Ho_{q-1}(\Delta_{I,m},\Bbbk)$ as claimed.  
\end{proof}

\mysection{Proof of Theorem \ref{stableTheorem}}\label{PRStableProof}
With the choice for the bases of the vector spaces in 
$\mathcal{D}(P_N)$ established, we now turn to the proof that the 
poset $P_N$ supports the minimal free resolution of $R/N$.  
We first analyze the action of the differential of 
$\mathcal{D}(P_N)$ when it is applied to an arbitrary 
basis element $f(I,m)$.  

\begin{lemma}\label{DifferentialAction}
The map $\varphi^{(I,m),(J,n)}_{q+1}$ sends a basic cycle $f(I,m)$ to 
the element $(-1)^{p+\delta_{m,n}}\cdot f(J,n)$, where $I=\{i_1,\ldots,i_q\}$, 
the relationship $I\smsm J=\{i_p\}$ holds and 
\begin{displaymath}
\delta_{m,n}=
\left\{
\begin{array}{ll}
1 & \textrm{if } m\ne n\\
0 & \textrm{otherwise.}
\end{array}
\right.
\end{displaymath} 
\end{lemma}

\begin{proof}
Write $d$ for the simplicial differential in 
the reduced chain complex $\CC(\Delta_{I,m})$.  
The open interval $\left((I,m),\hat{0}\right)$ 
may be realized as the union of half-closed intervals 
$\left[(J,n),\hat{0}\right)$, so that the order complex 
of each half-closed interval is a cone with apex $(J,n)$.  
Applying the differential to the sum of all facets 
contained in the interval produces the 
boundary of the cone, which in this case is the order 
complex of $\left((J,n),\hat{0}\right)$.  Indeed, 
when $d$ is applied to the sum 
$$
\ds\upsilon=\sum_{\sigma\in\left[(J,n),\hat{0}\right)}
\varepsilon_\sigma\cdot\sigma,
$$
the faces in which the element $(J,n)$ remains appear  
twice and have opposite signs as described in the proof of 
Lemma \ref{CycleLemma}.  Thus, the only faces that remain in the 
expansion of $d(\upsilon)$ are of the form 
$\bar{\sigma}=\sigma\smsm\{(J,n)\}$.  

Precisely, 
\begin{eqnarray}\label{phimapstable}
\varphi^{(I,m),(J,n)}_{q+1}(f(I,m)) 
& = &
\left[
d\left(
\ds\sum_{\sigma\in\left[(J,n),\hat{0}\right)}
\varepsilon_\sigma\cdot\sigma
\right)
\right]\\
& = & 
\left[
\ds\sum_{\sigma\in\left[(J,n),\hat{0}\right)}
\varepsilon_\sigma\cdot\bar{\sigma}
\right] \nonumber \\ 
& = & 
\left[
\ds\sum_{\bar{\sigma}\in\left((J,n),\hat{0}\right)}
\varepsilon_\sigma\cdot\bar{\sigma}
\right].\nonumber
\end{eqnarray}
 
The facet $\bar{\sigma}$ has an associated permutation 
$\rho_{\bar{\sigma}}\in\Sy_{q-1}$, and using elementary 
properties of permutation signs, we have 
$\sgn(\rho_{\sigma})=(-1)^{p+1}\cdot\sgn(\rho_{\bar{\sigma}})$, where $I\smsm J=\{i_p\}$.
Considering the definition of $\varepsilon_\sigma$, 
for each $(J,n)$ for which $(I,m)\ld(J,n)\in P_N^*$ we now have 
\begin{eqnarray*}
\varepsilon_\sigma 
& = & \sgn(\rho_\sigma)\cdot\sgn\left(\prod_{t=1}^q l_q\right)\\
& = & (-1)^{p+1}\cdot\sgn(\rho_{\bar\sigma})\cdot\sgn\left(\prod_{t=2}^q l_q\right)\cdot\sgn(l_1)\\
& = & (-1)^{p+1}\cdot\sgn(l_1)\cdot\varepsilon_{\bar\sigma}\\
& = & (-1)^{p+\delta_{m,n}}\cdot\varepsilon_{\bar{\sigma}}
\end{eqnarray*}
since $\sgn(l_1)=1$ if $n\ne m$ and $\sgn(l_1)=-1$ if $n=m$.  

Therefore, Equation (\ref{phimapstable}) becomes 
\begin{eqnarray*}
\varphi^{(I,m),(J,n)}_{q+1}(f(I,m))
& = &
\left[
\ds\sum_{\bar{\sigma}\in\left((J,n),\hat{0}\right)}
\varepsilon_\sigma\cdot\bar{\sigma}
\right]\\
& = &
\left[
\ds\sum_{\bar{\sigma}\in\left((J,n),\hat{0}\right)}
(-1)^{p+\delta_{m,n}}\cdot\varepsilon_{\bar{\sigma}}\cdot\bar{\sigma}
\right]\\
& = &
(-1)^{p+\delta_{m,n}}
\cdot\left[
\ds\sum_{\bar{\sigma}\in\left((J,n),\hat{0}\right)}
\varepsilon_{\bar{\sigma}}\cdot\bar{\sigma}
\right]\\
& = & 
(-1)^{p+\delta_{m,n}}\cdot f(J,n)
\end{eqnarray*}
which proves the lemma.
\end{proof}

As described in Section \ref{PRes}, the map $\varphi_{q+1}$ 
is defined componentwise on the one-dimensional $\Bbbk$-vectorspace 
$\mathcal{D}_{q+1,(I,m)}$ for each poset element $(I,m)$.  
Using the conclusion of Lemma \ref{DifferentialAction}, 
we immediately have  
\begin{equation}\label{PhiMap}
\ds\varphi_{q+1}|_{\mathcal{D}_{q+1,(I,m)}}=\varphi_{q+1,(I,m)}(f(I,m))=\sum_{(J,n)\ld(I,m)}(-1)^{p+\delta_{m,n}}f(J,n)
\end{equation}
where $I=\{i_1,\ldots,i_q\}$ and $i_1<\cdots<i_q$ and $J=I\smsm\{i_p\}$.  

Recall that the poset map $\eta:P_N\lra\mbb{N}^n$ 
is defined as $(I,m)\mapsto\mdeg(x_Im)$, so that 
we can homogenize the sequence of vector spaces 
$\mathcal{D}(P_N)$ to produce 
$$
  \mathcal{F}(\eta):  
  0\lra F_d
  \stackrel{\partial^{\mathcal{F}(\eta)}_d}{\lra}
  F_{d-1}\lra
  \ldots\lra 
  F_1\stackrel{\partial^{\mathcal{F}(\eta)}_1}{\lra}
  F_0,
$$ 
a sequence of multigraded modules.  
More precisely, for $q\ge 0$ and 
a poset element $(I,m)\ne\hat{0}$ 
where $I=\{i_1,\ldots,i_q\}$ and $i_1<\cdots<i_q$, 
the differential $\partial^{\mathcal{F}(\eta)}$ 
acts on a basis element $f(I,m)$ of the free 
module $F_{q+1}$ via the formula 

\begin{equation}\label{PResMap}
\begin{split}
\partial^{\mathcal{F}(\eta)}_{q+1}(f(I,m)) = \sum_{(J,n')\ld(I,m)}
 (-1)^{p+\delta_{m,n'}} x^{\eta(I,m)-\eta(J,n')}\cdot f(J,n') \phantom{spacerspacer}\\
  = \ds\sum_{(J,m)\ld(I,m)}(-1)^{p} x_{i_p}\cdot f(J,m) 
  - \ds\sum_{(J,n)\ld(I,m)}(-1)^{p} \frac{x_{i_p}m}{\g(x_{i_p} m)}\cdot f(J,n)
\end{split}  
\end{equation}

where $p$ takes the same value as in \ref{PhiMap}, so that $I\smsm\{i_p\}=J$.

It remains to show that $\mathcal{F}(\eta)$ is a minimal  
exact complex, and to do so we identify it as  
the Eliahou-Kervaire resolution.  

\begin{definition}\label{EKres}
The \emph{Eliahou-Kervaire minimal free resolution} \cite{EK}
of a stable ideal $N$ is  
$$\mathcal{E}:
  0\lra E_d
  \stackrel{\partial^\mathcal{E}_d}{\lra}
  E_{d-1}\lra
  \cdots\lra 
  E_1\stackrel{\partial^\mathcal{E}_1}{\lra}
  E_0
$$
where $E_0=R$ is the free module of rank one with basis $1$ 
and for $q\ge 0$, $E_{q+1}$ has as basis the admissible symbols
$$\left\{e(I,m): I=\{i_1,\ldots,i_q\},\max(I)<\max(m)\right\}.$$
When applied to a basis element, the differential of $\mathcal{E}$ 
takes the form 
\begin{eqnarray*}
  \partial^\mathcal{E}_{q+1}\left(e(I,m)\right) & = &
  \ds 
  \sum_{p=1}^q (-1)^p x_{i_p}\cdot e\left(I\smsm \{i_p\},m\right) \\
  & & -\sum_{p=1}^q (-1)^p \frac{x_{i_p} m}{\g(x_{i_p} m)}\cdot e\left(I\smsm \{i_p\},\g(x_{i_p} m)\right)
\end{eqnarray*}
Where we define $e(I\smsm\{i_p\},\g(x_p m))=0$ 
when $\max(I\smsm\{i_p\})\ge\max(\g(x_p m))$ 
(i.e. the symbol is inadmissible).   
\end{definition}

We now are in a position to prove the main result of this paper.  

\begin{proof}[Proof of Theorem \ref{stableTheorem}]
The Eliahou-Kervaire admissible symbols index the multigraded free modules in 
the complexes $\mathcal{E}$ and $\mathcal{F}(\eta)$ and 
therefore the generators of these modules are in one to one correspondence 
with one another.  
Further, comparing Definition \ref{EKres} and Equation \ref{PResMap}, 
$\partial^\mathcal{E}$ and $\partial^{\mathcal{F}(\eta)}$ 
have identical behavior on basis elements.  The minimality and exactness 
of $\mathcal{E}$ implies the minimality and exactness of $\mathcal{F}(\eta)$ 
so that $\mathcal{F}(\eta)$ is a minimal poset resoution of $R/N$.  
\end{proof}

\mysection{A Minimal Cellular Resolution of $R/N$}\label{StableCellular}

In this section we exhibit a minimal cellular 
resolution for an arbitrary stable monomial ideal.  
Cellular resolutions of monomial ideals have received 
considerable attention in the literature and recently 
Mermin \cite{Mermin} has shown that the Eliahou-Kervaire 
resolution $\mathcal{E}$ is cellular using methods distinct from 
those depicted here.  The techniques of Discrete 
Morse Theory have also been used by Batzies and Welker 
\cite{BatziesWelker} to produce a minimal cellular 
resolution of stable modules, which contain the class of 
stable ideals.  

More generally, methods for determining whether a given monomial  
ideal admits a minimal cellular (or CW) resolution 
remain an open question, although Velasco \cite{Velasco} has shown that 
there exist monomial ideals whose minimal free resolutions 
are not even supported on a CW complex.  

The technique described here is an example of a 
more general approach which interprets CW resolutions 
of monomial ideals through the theory of poset 
resolutions.  This approach is described in \cite{ClarkTchernev}, 
and is distinct from both the method of \cite{BatziesWelker} concerning 
stable modules and the method of \cite{Mermin} which is specific 
to stable ideals.  We begin by recalling a fundamental result due to Bj\"orner.  

\begin{proposition}\cite[Proposition 3.1]{BjCW}
A poset $P$ is a CW poset if and only if 
it is isomorphic to the face poset of a regular CW complex.
\end{proposition}

In the case of the poset of admissible symbols $P_N$, we 
interpret Bj\"orner's proof explicitly to produce the corresponding 
regular CW complex $X_N$.  
On the level of cells, $\hat{0}\in P_N$ corresponds 
to the empty cell and each admissible symbol $(I,m)$ of $P_N$ 
corresponds to a closed cell $X_{I,m}$ of dimension $|I|$ 
for which $P(X_{I,m})=[\hat 0,(I,m)]$.  Taking $X_N=\bigcup X_{I,m}$ 
we have an isomorphism of posets $P(X_N)\cong P_N$.  
The regular CW complex $X_N$ also comes 
equipped with a $\mbb{Z}^n$ grading by realizing the map 
$\eta:P_N\lra\mbb{N}^n$ of Theorem \ref{stableTheorem} as a map 
$\eta:X_N\lra\mbb{N}^n$ where a cell $X_{I,m}\mapsto\eta(I,m)=\mdeg(x_Im)$.

\begin{example}
The stable ideal $N=\langle a,b,c\rangle^2=\langle a^2,ab,ac,b^2,bc,c^2 \rangle$ has 
minimal resolution supported by $X_N$, the  
regular CW complex depicted below, which has six 0-cells, eight 1-cells and three 2-cells.  
The face poset of this cell complex $P(X_N)$ is isomorphic to the poset of admissible 
symbols $P_N$ given in Example \ref{PosetExample}.  

\centering
\includegraphics[scale=0.7]{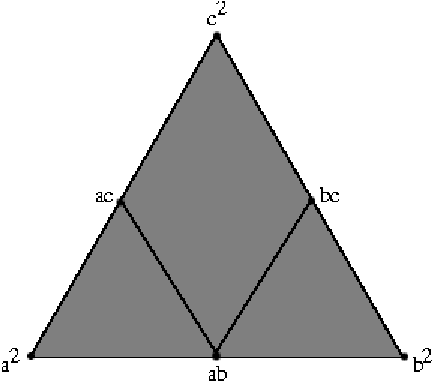}
\end{example}

We recall  the following well-known definition to which we incorporate 
the information given by the poset map $\eta$.  For a more comprehensive 
view of cellular and CW resolutions, see \cite{BatziesWelker,BS,Velasco}.

\begin{definition}\label{CellResDef}
A complex of multigraded $R$-modules, $\mathcal{F}_N$, is said to 
be a \emph{cellular resolution of $R/N$} if there 
exists an $\mbb{N}^n$-graded regular CW complex $X$ such that:
\begin{enumerate}
\item For all $i\ge{0}$, the free module $(\mathcal{F}_N)_i$ has as its 
      basis the $i-1$ dimensional cells of $X$.
\item For a basis element $e\in(\mathcal{F}_N)_i$, one has $\mdeg(e)=\eta(e)$,
\item The differential $\partial$ of \/ $\mathcal{F}_N$ acts on a basis 
      element $e\in (\mathcal{F}_N)_i$ as
      $$
        \partial(e)=
        \sum_{\substack{e'\subset{e}\subset{X} \\ \dim(e)=\dim(e')+1}}
        c_{e,e'}\cdot x^{\eta(e)-\eta(e')}\cdot e'
      $$
      where $c_{e,e'}$ is the coefficient of 
      the cell $e'$ in the differential of $e$ 
      in the cellular chain complex of $X$.  
\end{enumerate}
\end{definition}

With this definition in hand, we are now able to reinterpret 
Theorem \ref{stableTheorem} in our final result.  

\begin{theorem}\label{StableCellularTheorem}
Suppose that $N$ is a stable monomial ideal. 
Then the minimal free resolution $\mathcal{F}(\eta)$ 
is a minimal cellular resolution of $R/N$.
\end{theorem}

\begin{proof}
Conditions 1 and 2 of Definition \ref{CellResDef} are 
clear from the structure of $X_N$, its correspondence  
to the poset $P_N$ and the construction of the resolution 
$\mathcal{F}(\eta)$.  It therefore remains to verify that 
condition 3 is satisfied.  The main result 
in \cite{ClarkTchernev} provides a canonical isomorphism 
between the complex $\mathcal{D}(P_N)$ and $\mathcal{C}(X_N)$, 
the cellular chain complex of $X_N$.  Therefore, the 
differential of $\mathcal{F}(\eta)$ satisfies condition 3.  
\end{proof}

\end{document}